\documentclass[11pt]{article}
\usepackage{enumerate}
\usepackage{fullpage}
\usepackage{amsmath,amsthm,amssymb,graphicx,float}
\numberwithin{equation}{section}

\newtheorem{Th}{Theorem}[section]
\newtheorem{Lem}[Th]{Lemma}
\newtheorem{Def}[Th]{Definition}
\newtheorem{Cor}[Th]{Corollary}
\newtheorem{Rem}[Th]{Remark}
\newtheorem{Ex}[Th]{Example}

\newcommand{\R}{\mathbb R}

\def\R{\mathbb{R}}
\def\N{\mathbb{N}}

\def \sin {\hbox{\rm  sin }}
\def \cos {\hbox {\rm  cos }}

\begin{document}
\title{Stability Tests for Second Order Linear and Nonlinear Delayed Models}
\author{Leonid Berezansky \\
Dept. of Math, Ben-Gurion University of the Negev,\\ Beer-Sheva 84105, Israel \\
Elena Braverman\thanks{Corresponding author. E-mail {\em
maelena@math.ucalgary.ca}. Fax (403)-282-5150. Phone (403)-220-3956} \\
Dept. of Math \& Stats, University of Calgary, \\ 2500 University Dr. NW, Calgary, AB, Canada
T2N1N4 \\
and Lev Idels
\\
Dept. of Math,
   Vancouver Island University,\\
   900 Fifth St. Nanaimo, BC, Canada \, V9S5J5 }
\date{}
\maketitle


\begin{abstract}

For the nonlinear second order Lienard-type equations with time-varying delays
$$
\ddot{x}(t)+\sum_{k=1}^m f_k(t,x(t),\dot{x}(g_k(t)))+\sum_{k=1}^l s_k(t,x(h_k(t)))=0,
$$
global asymptotic  stability conditions are obtained. The results are based on the new
sufficient stability conditions for relevant linear equations and are applied to derive explicit  stability
conditions  for the nonlinear Kaldor-Kalecki business cycle model.
We also explore multistability of the sunflower non-autonomous equation and its modifications.
\end{abstract}

\noindent
{\bf AMS Subject Classification:} 34K20, 92D25, 34K45, 34K12, 34K25

\noindent
{\bf Keywords:} Second order delay differential equations; global asymptotic stability; boundedness of solutions;
Lienard-type nonautonomous linear and nonlinear delay differential equations; sunflower equation;
the Kaldor-Kalecki business cycle model; variable delays

\section{Introduction}

The second order delay differential equation
\begin{equation}
\label{0}
\displaystyle\ddot{x}+f(t,x(t),\dot{x}(t-\tau))+g(t, x(t), x(t-\tau))=0
\end{equation}
has a  more than 65-year history of study, and was used to examine
aftereffects in mechanics, physics, biology, medicine and economics (see, for example, \cite{Kol}).
Recently, these models have been used  to mimic regenerative vibrations in a milling
process, a balancing motion and chatter vibrations. For example, a one degree of freedom milling equation
\begin{equation}
\label{stepa}
\displaystyle\ddot{x}(t)+a\dot{x}(t)+bx(t)=-\alpha \left[ x(t)-x(t-\tau(t)) \right]
\end{equation}
was introduced in \cite{Xie}.
The milling model with several delays
\begin{equation}
\label{st3}
\displaystyle\ddot{x}(t)+a\dot{x}(t)+bx(t)+\sum_{k=1}^{p} \alpha_{k} [x(t)-x(t-\tau_{k})]^{k}=0\nonumber
\end{equation}
was recently studied, mostly numerically, in   \cite{stepa2,Kim,Liu}.
The following milling models with variable parameters were derived and examined in \cite{Kol,Gabor2, Gabor1,Xie,Wan,Yan}:
\begin{equation}
\label{L}
\displaystyle\ddot{x}(t)+a\dot{x}(t)+b(t)x(t)=c(t)x(t-\tau(t)),
\end{equation}
\begin{equation}
\label{sta}
\displaystyle\ddot{x}(t)+a\dot{x}(t)+b x(t)+\sum_{k=1}^{p} \alpha_{k}(t) [x(t)-x(t-\tau_{k}(t))]^{k}=0.\nonumber
\end{equation}
In economics, the well-known Kaldor-Kalecki business cycle model expressed as the delayed system of two nonlinear equations
\cite{Ga}, in some cases can be reduced to the second order equation (see, for example, \cite{Sz})
\begin{equation}\label{Busn}
\ddot{x}(t)+[\alpha-\beta p'(x(t))]\dot{x}(t)+\gamma [p(x(t))-\eta x(t)]+\delta p(x(t-\tau))=0.
\end{equation}
Here $p(x)$ is a frequently used in mathematical economics sigmoid function \cite{Ga},
e.g. $p(x)=\frac{A}{1+e^{-bx}}-\frac{A}{2}$, and all coefficients are nonnegative constants.

Different techniques were applied to study second-order delay equations  in
\cite{burt2,  Burt, caha, Gy, Liu2, Long, Pi} and \cite{Tun}--\cite{Zhang1}.
Characteristic quasipolynomials were broadly used for local stability analysis of autonomous models, (see, for
example, \cite{Kol}). The fixed point technique for second order differential and functional equations was pioneered by T. A. Burton
\cite{burt5,Bur6}.
In the paper \cite{Cahla} explicit and easily-verifiable tests were obtained for the autonomous model
\begin{equation}
\label{cah}
\ddot{x}(t)=p_1\dot{x}(t)+p_2\dot{x}(t-\tau)+q_1x(t)+q_2x(t-\tau).
\end{equation}

\begin{Th}\cite{Cahla}
Assume that at least one of the following conditions holds:
1) $p_1p_2>0, q_1>0,q_2>0$ or 2) $p_1>0,p_2>0, q_1>0, q_2<0$.
Then equation (\ref{cah}) is unstable.
\end{Th}
\begin{Th}\cite{Cahla}
\label{th_cahl2}
Assume $p_1=p_2=0$, $q_2>0$ and denote $B=\tau^2q_1$, $D=\tau^2q_2$.
Equation (\ref{cah}) is asymptotically stable if and only if $q_1<0$ and
there exists $k \in \N$ such that
$$
2k\pi<\sqrt{-B}<(2k+1)\pi,~ D<\min\left\{ -(2k)^2\pi^2-B,(2k+1)^2\pi^2+B \right\}.
$$
\end{Th}
\begin{Ex}
The second-order delay equation
\begin{equation}\label{Ex1}
\ddot{x}(t)=-49x(t)+7x(t-1)
\end{equation}
is asymptotically stable by Theorem~\ref{th_cahl2}.
Based on the algorithmic tests presented in \cite{Cahla}, the equation
\begin{equation}\label{Ex2}
\ddot{x}(t)=0.6\dot{x}(t)+0.3\dot{x}(t-1)-2x(t)+x(t-1)
\end{equation}
is asymptotically stable.
It is interesting to note that equations (\ref{Ex1}) and (\ref{Ex2}) without delays are unstable.
This illustrates a very interesting feature of second-order delay differential equations, i.e. delays may improve asymptotic
properties of a given equation, whereas delays in first-order linear equations  have mostly destabilizing effects or do not change stability
of the model.
\end{Ex}
Several stability tests for non-autonomous linear models with  variable delays
\begin{equation}
\label{1}
\ddot{x}(t) +a(t)\dot{x}(g(t))+b(t)x(h(t))=0,
\end{equation}
\begin{equation}
\label{2}
\ddot{x}(t) +a(t)\dot{x}(t)+b(t)x(t)+a_1(t)\dot{x}(g(t))+b_1(t)x(h(t))=0,
\end{equation}
were obtained in our recent paper \cite{B2}, under the assumptions:
$a$, $a_1$, $b$ and $b_1$ are Lebesgue measurable and
essentially bounded functions on $[0,\infty)$;
$a(t)\geq a_0>0$, $b(t)\geq b_0>0$, $0\leq t-h(t)\leq \tau$, $0\leq t-g(t)\leq \delta$,
$a^2(t)\geq 4b(t)$,  $\int_{g(t)}^t a(s)ds<1/e$.
Below $\| \cdot \|$ is the norm in the space ${\bf L}_{\infty}[t_0,\infty)$.

\begin{Th}\cite[Theorem 5.1]{B2}
If for some $t_0\geq 0$
\begin{equation}
\label{3}
\delta\left\|\frac{a}{b}\right\|\left(\|a \|\left\|\frac{b}{a}\right\|+
\|b\| \right)+\tau\left\|\frac{b}{a}\right\|<1,\nonumber
\end{equation}
then equation (\ref{1}) is exponentially stable.
\end{Th}
\begin{Th} \cite[Theorem 5.3]{B2}
Suppose for some $t_0\geq 0$
$$
\left\|\frac{a_1}{a}\right\|<1,~ \left\|\frac{a_1}{b}\right\|\frac{\left\|\frac{b}{a}\right\|+
\left\|\frac{b_1}{a}\right\|}
{1-\left\|\frac{a_1}{a}\right\|}+\left\|\frac{b_1}{b}\right\|<1,
$$
then equation (\ref{2}) is exponentially stable.
\end{Th}

In the present paper, a specially designed substitution  transforms  linear second order  equations into
a system, with a further application of the M-matrix method. This and the linearization techniques are used
to devise new global stability tests for nonlinear non-autonomous models.
These results are explicit, easily verifiable and can be applied to a general class of second order
non-autonomous equations. Some of the theorems of the present paper complement our earlier results
\cite{B1, B2}, as well as the tests obtained in recent papers \cite{Cahla, caha, Gy}.

The paper is organized as follows.
Section 2 contains stability results for linear second order non-autonomous equations with several delays.
To illustrate efficiency of the results obtained each stability test is accompanied by numerical examples.
In  Section 3 the tests for linear models are applied to nonlinear Lienard-type equations of the second order.
Applications incorporate a global stability test for the non-autonomous business cycle model.
Section 4 includes the study of bounds and multistability properties for the sunflower model
and its generalizations. In particular, sufficient conditions for convergence to one of an infinite number of
equilibrium points are presented, and existence of unbounded linearly growing solutions is illustrated.
Final remarks are presented in Section 5.

\section{Stability tests for linear Lienard equations}

The technique in this section involves parlaying a second order equation into two first order equations.
Consider a linear equation of the second order
\begin{equation}\label{general}
\ddot{x}(t)+\sum_{k=1}^m a_k(t)\dot{x}(h_k(t))+\sum_{k=1}^m b_k(t)\int_{g_k(t)}^t \dot{x}(s)ds
+\sum_{k=1}^m c_k(t)x(r_k(t))=0.
\end{equation}
Together with equation (\ref{general}),  for any $t_0\geq 0$ we consider the initial condition
\begin{equation}\label{in}
x(t)=\varphi(t), ~\dot{x}(t)=\psi(t), ~t\leq t_0.
\end{equation}
Henceforth, we assume that the following assumptions are satisfied:

(a1) $a_i, b_i,  c_i, i=1,\dots,m$ are Lebesgue measurable and essentially bounded on $[0,\infty)$;

(a2) $h_i, g_i, r_i$
are  Lebesgue measurable functions,
$
h_i(t)\leq t, g_i(t)\leq t, r_i(t)\leq t$, \\ $\lim\limits_{t\to\infty} h_i(t)=\infty$, $\lim\limits_{t\to\infty}
g_i(t)=\infty, \lim_{t\to\infty} r_i(t)=\infty$,
$i,j=1,\dots,m $;

(a3) $\varphi$ and $ \psi$ are Borel measurable bounded functions.

\begin{Def}A function $x:\R \rightarrow \R$ with locally
absolutely continuous on $[t_0,\infty)$ derivative $\dot{x}$ is called
{\bf a solution} of  problem (\ref{general}), (\ref{in})
if it satisfies equation (\ref{general}) for almost every
$t\in [t_0, \infty)$ and equalities (\ref{in}) for $t\leq t_0$.
\end{Def}
We quote a useful lemma that will play a major role in the proofs.
\begin{Lem}\label{lemma1} \cite{BBI_submitted}
Consider the system
\begin{equation}
\label{system_1}
\dot{x_i}(t)=-a_i(t)x_i(t)+\sum_{j=1}^m\sum_{k=1}^{l_{ij}}b_{ij}^k(t)x_j(h_{ij}^k(t)), ~~i=1,\dots,m,
\end{equation}
where
${\displaystyle
a_i(t)\geq \alpha_i>0, ~|b_{ij}^k(t)|\leq L_{ij}^k, ~ t-h_{ij}^k(t)\leq \sigma_{ij}^k.
}$

If the matrix $\displaystyle B=(b_{ij})_{i,j=1}^m$, with
$\displaystyle b_{ii}= \left. 1-\left( \sum_{k=1}^{l_{ii}}L_{ii}^k \right) \right/ \alpha_i$,
$\displaystyle b_{ij}=- \left. \left( \sum_{k=1}^{l_{ij}}L_{ij}^k\right) \right/ \alpha_i$, $i\neq j$,
is an M-matrix, then system (\ref{system_1}) is exponentially stable.
\end{Lem}

We recall that a matrix $B=(b_{ij})_{i,j=1}^m$ is a (nonsingular) {\em $M$-matrix} if $b_{ij}\leq 0$, $i\neq j$ and
one of the following equivalent conditions holds: either there exists a positive inverse matrix $B^{-1} > 0$
or all the principal minors of the matrix $B$ are positive.

Further proofs will also require the following lemma.

\begin{Lem}
\label{lemma2}
Consider the system
\begin{equation}
\label{system_2}
\dot{x_i}(t)=-a_i(t)x_i(t)+\sum_{j=1}^m\sum_{k=1}^{l_{ij}}\left( c_{ij}^k (t)  x_j(g_{ij}^k(t))+d_{ij}^k(t)
\int_{h_{ij}^k(t)}^t x_j(s)ds\right), ~~i=1,\dots,m,
\end{equation}
where
$\displaystyle
a_i(t)\geq \alpha_i>0,~ |d_{ij}^k(t)|\leq L_{ij}^k, ~|c_{ij}^k(t)|\leq C_{ij}^k,
~t-h_{ij}^k(t)\leq \sigma_{ij}^k, ~ t-g_{ij}^k(t)\leq \tau.
$
If the matrix $\displaystyle B=(b_{ij})_{i,j=1}^m$, with
$\displaystyle b_{ii}=1-\sum_{k=1}^{l_{ii}}\left(L_{ii}^k\sigma_{ii}^k+
C_{ii}^k\right)/\alpha_i$,
$\displaystyle b_{ij}=-\sum_{k=1}^{l_{ij}}\left(L_{ij}^k\sigma_{ij}^k+ C_{ij}^k\right)/\alpha_i$,
$i\neq j$, is an M-matrix, then system (\ref{system_2}) is exponentially stable.
\end{Lem}
\begin{proof}
Let $x(t)$ be a solution of (\ref{system_2}). Since $x_j(t)$ are continuous then for any $i,j,k$ and $t$ there exists
$p_{ij}^k(t)\in (h_{ij}^k(t),t)$ such that $\displaystyle x_j(p_{ij}^k(t))(t-h_{ij}^k(t))=\int_{h_{ij}^k(t)}^t x_j(s)ds$.

Thus $x_j$ are solutions of system (\ref{system_1}) with $b_{ij}^k(t)x_j(h_{ij}^k(t))$ being replaced by
$c_{ij}^k x_j(g_{ij}^k(t))+d_{ij}^k(t)(t-h_{ij}^k(t))x_j(p_{ij}^k(t))$.
We have $|c_{ij}^k(t)|\leq C_{ij}^k,~|d_{ij}^k(t)(t-h_{ij}^k(t))|\leq L_{ij}^k\sigma_{ij}^k, ~i\neq j$.
The application of Lemma~\ref{lemma1}
validates the proof.
\end{proof}
{\bf Note that a different proof of Lemma~\ref{lemma2} involves application of the Halanay-type inequalities (see, for example, \cite{Liz}).}
\\
To examine the equation
\begin{equation}\label{7}
\ddot{x}(t)+a(t)\dot{x}(t)+b(t)x(t)+\sum_{k=1}^m c_k(t)x(h_k(t))=0
\end{equation}
we assume
$$
0<a\leq a(t)\leq A,~ 0<b\leq b(t)\leq B,~ |c_k(t)|\leq C_k,~ t-h_k(t)\leq \tau.
$$

\begin{Th}\label{th1} Suppose at least one of the following conditions holds:\\
1) $\displaystyle B\leq\frac{a^2}{4},~ \sum_{k=1}^m C_k<b-\frac{a}{2}(A-a)$,\\
2) $\displaystyle b\geq\frac{a}{2}\left(A-\frac{a}{2}\right),~ \sum_{k=1}^m C_k<\frac{a^2}{2}-B$.\\
Then equation (\ref{7}) is exponentially stable.
\end{Th}
\begin{proof}
Substituting $\displaystyle \dot{x}=-\frac{a}{2} x+y,
\ddot{x}=-\frac{a}{2}\dot{x}+\dot{y}$  into equation (\ref{7}),
we arrive at
\begin{equation}\label{8}
\begin{array}{l}
\displaystyle \dot{x}=-\frac{a}{2} x+y\\
\displaystyle \dot{y}=\left[\frac{a}{2}\left(a(t)-\frac{a}{2}\right)-b(t)\right]x(t)-\sum_{k=1}^m
c_k(t)x(h_k(t))-\left(a(t)-\frac{a}{2}\right)y(t).
\end{array}
\end{equation}
Condition 1) yields
$\displaystyle
\frac{a}{2}\left(a(t)-\frac{a}{2}\right)-b(t)\geq \frac{a^2}{4}-B \geq 0,
$  $\displaystyle
 \frac{a}{2}\left(a(t)-\frac{a}{2}\right)-b(t)\leq \frac{a}{2}\left(A-\frac{a}{2}\right)-b.
$
Hence the matrix
$$
\left(\begin{array}{cc}
1~~~~~~~~~~~~~~~~~&-\frac{2}{a}\\
\displaystyle  -\frac{2}{a} \left( \frac{a}{2}\left(A-\frac{a}{2}\right)-b+\sum_{k=1}^m  C_k \right) &~1
\end{array}\right)
$$
is an M-matrix. By Lemma~\ref{lemma1}  equation (\ref{7}) is exponentially stable.

If condition 2) holds then
$\displaystyle
b(t)-\frac{a}{2}\left(a(t)-\frac{a}{2}\right)\geq  b-\frac{a}{2}\left(A-\frac{a}{2}\right) \geq 0,
$ \\
$\displaystyle  b(t)-\frac{a}{2}\left(a(t)-\frac{a}{2}\right)\leq
B-\frac{a}{2}\left(a-\frac{a}{2}\right)=B-a^2/4,
$
and the matrix
$$
\left(\begin{array}{cc}
1~~~~~~~~~~~~~~~~&-\frac{2}{a}\\
\displaystyle -\frac{2}{a} \left( B-\frac{a^2}{4}+\sum_{k=1}^m C_k \right) &~1
\end{array}\right)
$$
is an M-matrix. By Lemma~\ref{lemma1}  equation (\ref{7}) is exponentially stable.
\end{proof}

\begin{Rem}
{\bf
Application of the classical substitution $\dot{x}=y$ is not useful in our stability investigation,
since for the system obtained after this substitution, the matrix $B$ in Lemma~\ref{lemma1} is not an M-matrix.
For equation (\ref{7}) with constant $a$ and $b$, $|C(t)|\leq C$  and $m=1$
\begin{equation}\label{t}
\ddot{x}(t)+a\dot{x}(t)+b(t)x(t)+c(t)x(h(t))=0,
\end{equation}
we compare two substitutions
\begin{equation}\label{s}
\dot{x}(t)=-\lambda x(t)+y(t),~ \lambda>0,
\end{equation}
and $\displaystyle \dot{x}(t)=-\frac{a}{2} x(t)+y(t)$.
By Theorem~\ref{th1} equation \eqref{t} is exponentially stable, if at least one of the following conditions holds:

1) $b\leq \frac{a^2}{4},  C<b,$

2) $b>\frac{a^2}{4}, C<\frac{a^2}{2}-b$.

Whereas application of \eqref{s} by the same token yields a slight improvement:
equation \eqref{t} is exponentially stable, if at least one of the following conditions holds:

1), 2) or

3)  $b\leq \frac{a^2}{4}, C<\frac{a^2}{2}-b$ and the following two intervals have a nonempty intersection
$$
\left[\frac{a-\sqrt{a^2-4b}}{2},\frac{a+\sqrt{a^2-4b}}{2}\right] \cap \left[ a-\sqrt{a^2-2(b+C)},a+\sqrt{a^2-2(b+C)}
\right] \neq \emptyset.
$$
Implementation of \eqref{s} for equation (\ref{7}) with
nonconstant coefficients $a(t)$ and $ b(t)$ will produce a more complicated condition 3).
Trading-off these options, we
prefer the substitution $\dot{x}(t)=-\frac{a}{2} x(t)+y(t)$.
}
\end{Rem}

{\bf
The following numerical examples illustrate the application of Theorem~\ref{th1}. }
\begin{Ex}
Consider the delay equation
\begin{equation}\label{9}
\ddot{x}(t)+a\dot{x}(t)+bx(t)+cx(t-h|\sin t|)=0
\end{equation}

a) $a=3, b=1.1, c=-0.8, h=2$. Condition 1) of Theorem~\ref{th1} holds, condition 2) does not hold. Equation (\ref{9})
is asymptotically stable.\\
b) $a=2, b=1.1, c=-0.8, h=2$. Condition 2) of Theorem~\ref{th1} holds, condition 1) does not hold. Equation (\ref{9})
is asymptotically stable.\\
c) $a=0.1, b=1.5, c=-1.45, h=2$. Conditions of Theorem~\ref{th1}  do not hold, and equation (\ref{9})
is unstable. 

{\bf
Let us note that in c) the coefficient of the non-delay term exceeds the one of the delayed term:
$|c(t)|<b(t)$. This is in contrast to the result for the equation with the second derivative omitted
$$
a\dot{x}(t)+b(t)x(t)+c(t)x(h(t))=0,
$$
which is exponentially stable if $a>0$, $b(t)\geq b_0>0$, $|c(t)|<b(t)$, $t-\tau \leq h(t) \leq t$ for $\tau>0$.
}
\end{Ex}

Consider the equation
\begin{equation}\label{10}
\ddot{x}(t)+a(t)\dot{x}(t)+b(t)x(t)+\sum_{k=1}^m c_k(t)\dot{x}(h_k(t))=0,
\end{equation}
where
$$
0<a\leq a(t)\leq A,~ 0<b\leq b(t)\leq B, ~|c_k(t)|\leq C_k,~ t-h_k(t)\leq \tau.
$$
\begin{Th}\label{th2} Suppose  that   at least one of the following conditions holds:\\
1) $\displaystyle B\leq \frac{a^2}{4},~ \sum_{k=1}^m C_k<\frac{2b-a(A-a)}{2a}$,\\
2) $b\displaystyle \geq \frac{a}{2}\left(A-\frac{a}{2}\right),~ \sum_{k=1}^m C_k<\frac{a^2-2B}{2a}$.\\
Then equation (\ref{10}) is exponentially stable.
\end{Th}
\begin{proof}
The substitution $\dot{x}=- \frac{a}{2}x+y, \ddot{x}=-\frac{a}{2}\dot{x}+\dot{y}$ into equation (\ref{10})
yields
\begin{equation}\label{11}
\begin{array}{ll}
\displaystyle \dot{x}=& \displaystyle  -\frac{a}{2} x+y\\
\displaystyle \dot{y}=& \displaystyle  \left[\frac{a}{2}\left(a(t)-\frac{a}{2}\right)-b(t)\right]x(t)+\frac{a}{2}\sum_{k=1}^m c_k(t)x(h_k(t))\\
& \displaystyle -\sum_{k=1}^m c_k(t)y(h_k(t))
-\left(a(t)-\frac{a}{2}\right)y(t).
\end{array}
\end{equation}
If condition 1) holds, we have
$\displaystyle
\frac{a}{2}\left(a(t)-\frac{a}{2}\right)-b(t)\geq \frac{a^2}{4}-B \geq 0,
$ \\ \\ $
\displaystyle \frac{a}{2}\left(a(t)-\frac{a}{2}\right)-b(t)\leq \frac{a}{2}\left(A-\frac{a}{2}\right)-b.
$
Hence the matrix
$$
\left(\begin{array}{cc}
1  &-\frac{2}{a}\\
\displaystyle -\frac{2}{a} \left[ \frac{a}{2}\left(A-\frac{a}{2}\right)-b+\frac{a}{2}\sum_{k=1}^m C_k \right]
& \displaystyle 1-\frac{2}{a}\sum_{k=1}^m C_k
\end{array}\right)
$$
is an M-matrix. By Lemma~\ref{lemma1}  equation (\ref{10}) is exponentially stable.

If the inequalities in 2) hold then
$\displaystyle
b(t)-\frac{a}{2}\left(a(t)-\frac{a}{2}\right)\geq  b-\frac{a}{2}\left(A-\frac{a}{2}\right) \geq 0,
$ \\$
\displaystyle b(t)-\frac{a}{2}\left(a(t)-\frac{a}{2}\right)\leq
B-\frac{a}{2}\left(a-\frac{a}{2}\right)=B-a^2/4.
$
Thus the matrix
$$
\left(\begin{array}{cc}
1& -\frac{2}{a} \\
\displaystyle -\frac{2}{a} \left[ B-\frac{a^2}{4}+\frac{a}{2}\sum_{k=1}^m C_k \right] &
\displaystyle 1-\frac{2}{a}
\sum_{k=1}^m C_k
\end{array}\right)
$$
is an M-matrix. By Lemma~\ref{lemma1}  equation (\ref{10}) is exponentially stable.
\end{proof}

\begin{Ex}
Consider the equation
\begin{equation}\label{12_ex}
\ddot{x}(t)+a\dot{x}(t)+bx(t)+c\dot{x}(t-h|\sin t|)=0
\end{equation}
To illustrate Theorem~\ref{th2}, we  examined:\\
a) $a=2.1$, $b=1$, $c=-0.4$, $h=2$. Condition 1) of Theorem~\ref{th2} holds, condition 2) does not hold. Equation (\ref{12_ex})
is asymptotically stable.\\
b) $a=4$, $b=5$, $c=-0.7$, $h=2$. Condition 2) of Theorem~\ref{th2} holds, condition 1) does not hold. Equation (\ref{12_ex})
is asymptotically stable.\\
c) $a=1$, $b=1.5$, $c=-0.8$, $h=2$. Conditions of the Theorem~\ref{th2}  do not hold, and,  as can be seen from numerical
simulations, equation (\ref{12_ex})
is unstable. Hence, in general, the conditions $a(t)\geq a_0>0, b(t)\geq b_0>0, m=1, |c(t)|<a(t)$
are not sufficient for stability of equation (\ref{10}).
\end{Ex}

Consider the equation
\begin{equation}\label{12}
\ddot{x}(t)+a(t)\dot{x}(t)+\sum_{k=1}^m b_k(t)x(h_k(t))=0,
\end{equation}
where $0<a\leq a(t)\leq A$, $0<b_k\leq b_k(t)\leq B_k,~ t-h_k(t)\leq \tau_k$.

\begin{Th}\label{th3}
Suppose  at least one of the following conditions holds:\\
1) $\displaystyle \sum_{k=1}^m B_k\leq \frac{a^2}{4},~ \frac{a}{2}(A-a)<\sum_{k=1}^m b_k-a\sum_{k=1}^m B_k\tau_k$,\\
2) $\displaystyle \sum_{k=1}^m b_k\geq \frac{a}{2}\left(A-\frac{a}{2}\right),~ \sum_{k=1}^m
B_k\left(1+a\tau_k\right)<\frac{a^2}{2}$.\\
Then equation (\ref{12}) is exponentially stable.
\end{Th}
\begin{proof}
With the substitution $\dot{x}=-\frac{a}{2} x+y,
\ddot{x}=-\frac{a}{2}\dot{x}+\dot{y}$ into equation (\ref{12}),
we arrive at
\begin{equation}\label{13}
\begin{array}{ll}
\displaystyle \dot{x}= & \displaystyle -\frac{a}{2} x+y \\
\displaystyle \dot{y}= & \displaystyle \left[\frac{a}{2}\left(a(t)-\frac{a}{2}\right)-\sum_{k=1}^m b_k(t)\right]x(t)\\
& \displaystyle +\sum_{k=1}^m b_k(t)\int_{h_k(t)}^t \left[-\frac{a}{2} x(s)+y(s)\right]ds-\left(a(t)-\frac{a}{2}\right)y(t).
\end{array}
\end{equation}
If condition 1) holds, we have
$\displaystyle
\frac{a}{2}\left(a(t)-\frac{a}{2}\right)-\sum_{k=1}^m b_k(t)\geq \frac{a^2}{4}-\sum_{k=1}^m B_k \geq 0,
$ \\ $\displaystyle
\frac{a}{2}\left(a(t)-\frac{a}{2}\right)-\sum_{k=1}^m b_k(t)\leq \frac{a}{2}\left(A-\frac{a}{2}\right)-\sum_{k=1}^m b_k.
$
Hence the off-diagonal entries of the matrix
$$
\left(\begin{array}{cc}
1~~~~~~~~~~~~~~~~~&-\frac{2}{a} \\
\displaystyle  -\frac{2}{a} \left[ \frac{a}{2}\left(A-\frac{a}{2}\right)-\sum_{k=1}^m
b_k+\frac{a}{2}\sum_{k=1}^m B_k\tau_k \right] & \displaystyle ~1-\frac{2}{a}\sum_{k=1}^m B_k\tau_k
\end{array}\right)
$$
are non-positive, and the inequalities in 1) yield that it is an M-matrix. By Lemma~\ref{lemma2}  equation (\ref{12}) is exponentially
stable.
Assumption 2) implies
$\displaystyle
\sum_{k=1}^m b_k(t)-\frac{a}{2}\left(a(t)-\frac{a}{2}\right)\geq  \sum_{k=1}^m b_k-\frac{a}{2}\left(A-\frac{a}{2}\right) \geq 0,
$  $\displaystyle
\sum_{k=1}^m b_k(t)-\frac{a}{2}\left(a(t)-\frac{a}{2}\right)\leq \sum_{k=1}^m B_k-\frac{a}{2}\left(a-\frac{a}{2}\right)=\sum_{k=1}^m B_k-a^2/4,
$
therefore the matrix
$$
\left(\begin{array}{cc}
1~~~~~~~~~~~~~~~~&-\frac{2}{a} \\
\displaystyle -\frac{2}{a} \left[ \sum_{k=1}^m B_k-\frac{a^2}{4}+\frac{a}{2}\sum_{k=1}^m B_k\tau_k \right]
&\displaystyle ~1- \frac{2}{a}\sum_{k=1}^m B_k\tau_k
\end{array}\right)
$$
is an M-matrix. By Lemma~\ref{lemma2}  equation (\ref{12}) is exponentially stable.
\end{proof}

\begin{Cor}
Suppose $a(t)\equiv a>0, b_k(t)\equiv b_k>0$, and at least one of the following conditions holds:\\
1) $\displaystyle \sum_{k=1}^m b_k\leq\frac{a^2}{4},~ \sum_{k=1}^m b_k(1-a\tau_k) >0$,\\
2) $\displaystyle \sum_{k=1}^m b_k\geq \frac{a^2}{4},~ \sum_{k=1}^m b_k(1+a\tau_k)<\frac{a^2}{2}$.\\
Then equation (\ref{12}) is exponentially stable.
\end{Cor}

\begin{Ex}
Consider the equation
 \begin{equation}\label{14}
\ddot{x}(t)+a\dot{x}(t)+bx(t-h|\sin t|)=0.
\end{equation}
To illustrate  Theorem~\ref{th3}, we consider numerical examples:\\
a) $a=2$, $b=0.9$, $h=0.4$. Condition 1) of Theorem~\ref{th3} holds, condition 2) does not hold. Equation (\ref{14})
is asymptotically stable.\\
b) $a=2$, $b=1.1$,  $h=0.4$. Condition 2) of Theorem~\ref{th3} holds, condition 1) does not hold. Equation (\ref{14})
is asymptotically stable.\\
c) $a=1$, $b=1.1$, $h=2.5$. Conditions of Theorem~\ref{th3}  do not hold. Equation (\ref{13})
is unstable which can be confirmed numerically.

\end{Ex}

Consider the equation
\begin{equation}\label{15}
\ddot{x}(t)+a(t)\dot{x}(t)+b(t)x(t)=\sum_{k=1}^m c_k(t) \left[ x(t)-ix(h_k(t)) \right],
\end{equation}
where $0<a\leq a(t)\leq A,~ 0<b\leq b_k(t)\leq B,~ |c_k(t)|\leq C_k,~ t-h_k(t)\leq \tau_k$.

\begin{Th}\label{th4} Suppose at least one of the following conditions holds:\\
1) $\displaystyle B\leq\frac{a^2}{4},~ \sum_{k=1}^m C_k\tau_k<\frac{2b-a(A-a)}{2a}$,\\
2) $\displaystyle b\geq\frac{a}{2}\left(A-\frac{a}{2}\right),~ \sum_{k=1}^m C_k\tau_k<\frac{a^2-2B}{2a}$.\\
Then equation (\ref{15}) is exponentially stable.
\end{Th}
\begin{proof}
After rewriting equation~(\ref{15}) in the form
$$
\ddot{x}(t)+a(t)\dot{x}(t)+b(t)x(t)=\sum_{k=1}^m c_k(t)\int_{h_k(t)}^t \dot{x}(s)ds,
$$
we apply the same argument as in the proof of Theorem~\ref{th2}.
\end{proof}

Theorem~\ref{th1} gives delay-independent stability conditions for equation~(\ref{7}).
The following statement contains delay-dependent stability conditions for this equation.

\begin{Th}\label{th5}
Assume that
$$
0<a\leq a(t)\leq A, 0<b\leq b(t)+\sum_{k=1}^m c_k(t)\leq B, |c_k(t)|\leq C_k, t-h_k(t)\leq \tau_k
$$
and at least one of the conditions of Theorem~\ref{th4} holds.
Then equation (\ref{7}) is exponentially stable.
\end{Th}
\begin{proof}
Rewrite equation~(\ref{7}) in the form
$$
\ddot{x}(t)+a(t)\dot{x}(t)+\left(b(t)+\sum_{k=1}^m c_k(t)\right)x(t)=\sum_{k=1}^m c_k(t)\int_{h_k(t))}^t \dot{x}(s)ds.
$$
The end of the proof is a straightforward imitation of the proof of Theorem~\ref{th2}.
\end{proof}

\section{Stability tests for nonlinear Lienard equations}

In this section we examine  several nonlinear delay differential equations of the second order which have the following general form
\begin{equation}\label{16}
\ddot{x}(t)+\sum_{k=1}^m f_k(t,x(p_k(t)),\dot{x}(g_k(t)))+\sum_{k=1}^l s_k(t,x(h_k(t)))=0,
\end{equation}
with the following initial function
\begin{equation}\label{17}
x(t)=\varphi(t),\dot{x}(t)=\psi(t), t\leq t_0,~ t_0\geq 0
\end{equation}
where $f_k(t,u_1,u_2), k=1,\dots,m,~ s_k(t,u),$ are Caratheodory functions which are
measurable in $t$ and continuous in all the other arguments,
condition (a2) holds for delay functions  $p_k, g_k, h_k$; $\varphi$ and $\psi$ are Borel measurable bounded functions.

The definition of the solution of the initial value problem (\ref{16})-(\ref{17}) is the same as for problem (\ref{general}), (\ref{in}).
We will assume that the initial value problem has a unique global solution on $[t_0,\infty)$ for all nonlinear equations considered in this section.

{\bf
\begin{Def}
Suppose the number $K$ is an equilibrium of equation (\ref{16}). We will say that
$K$ is {\em an attractor} of this equation if for any solution $x$ of the problem
(\ref{16}), (\ref{17}) we have $\lim_{t\rightarrow\infty}x(t)=K$.
\end{Def}
}

\begin{Th}
 Consider the equation
\begin{equation}\label{19}
\ddot{x}(t)+f(t,x(t),\dot{x}(t))+s(t,x(t))+\sum_{k=1}^m s_k(t,x(t),x(h_k(t)))=0,
\end{equation}
where
$$
f(t,v,0)=0, s(t,0)=0, s_k(t,v,0)=0, 0<a_0\leq \frac{f(t,v,u)}{u}\leq A,
$$$$
0<b_0\leq \frac{s(t,u)}{u}\leq B, \left| \frac{s_k(t,v,u)}{u} \right| \leq C_k, u\neq 0,~t-h_k(t)\leq \tau.
$$
If at least one of the following conditions holds:\\
1) $\displaystyle B\leq \frac{a_0^2}{4},~ \sum_{k=1}^m C_k<b_0-\frac{a_0}{2}(A-a_0)$,\\
2) $\displaystyle b_0\geq \frac{a_0}{2}\left(A-\frac{a_0}{2}\right),~ \sum_{k=1}^m C_k<\frac{a_0^2}{2}-B$,\\
then zero is a global  attractor for all solutions of problem~(\ref{19}), (\ref{17}).
\end{Th}
\begin{proof}
{\bf Suppose $x$ is a fixed solution of problem~(\ref{19}), (\ref{17}).}
Rewrite equation~(\ref{19}) in the form
$$
\ddot{x}(t)+a(t)\dot{x}(t)+b(t)x(t)+\sum_{k=1}^m c_k(t)x(h_k(t))=0,
$$
where
$\displaystyle
a(t)= \left\{\begin{array}{cc}
\frac{f(t,x(t),\dot{x}(t))}{\dot{x}(t)},& \dot{x}(t)\neq 0,\\
a_0,& \dot{x}(t)=0,\end{array}\right. ~~
b(t)=\left\{\begin{array}{cc}
\frac{s(t,x(t))}{x(t)},& x(t)\neq 0,\\
b_0,& x(t)=0,\end{array}\right.
$ \\ $\displaystyle
c_k(t)= \left\{\begin{array}{cc}
\frac{s_k(t,x(t),x(h_k(t)))}{x(h_k(t))},& x(h_k(t))\neq 0,\\
0,&  x(h_k(t))=0.\end{array}\right.
$

Hence the function $x$ is a solution of the linear equation
\begin{equation}\label{20}
\ddot{y}(t)+a(t)\dot{y}(t)+b(t)y(t)+\sum_{k=1}^m c_k(t)y(h_k(t))=0,
\end{equation}
which is exponentially stable by  Theorem~\ref{th1}.
Thus for any solution
$y$ of equation (\ref{20}) we have $\lim\limits_{t\rightarrow\infty}y(t)=0$. Since $x$ is a solution of (\ref{20}),
we have  $\lim\limits_{t\rightarrow\infty}x(t)=0.$
\end{proof}

The previous proof is readily adapted to the proof of the following theorems.
\begin{Th}\label{th3.3}
Consider the equation
\begin{equation}\label{21}
\ddot{x}(t)+f(t,x(t),\dot{x}(t))+s(t,x(t))+\sum_{k=1}^m s_k(t,x(t),\dot{x}(h_k(t)))=0,
\end{equation}
where
$$
f(t,v,0)=0,~ s(t,0)=0,~ s_k(t,v,0)=0,~ 0<a_0\leq \frac{f(t,v,u)}{u}\leq A,
$$
$$
0<b_0\leq \frac{s(t,u)}{u}\leq B, ~ \left|\frac{s_k(t,v,u)}{u}\right|\leq C_k, u\neq 0,~t-h_k(t)\leq \tau.
$$
Suppose at least one of the following conditions holds:\\
1) $\displaystyle B\leq \frac{a_0^2}{4},~ \sum_{k=1}^m C_k<\frac{2b_0-a_0(A-a_0)}{2a_0}$,\\
2) $\displaystyle b_0 \geq \frac{a_0}{2}\left(A-\frac{a_0}{2}\right),~ \sum_{k=1}^m
C_k<\frac{a_0^2-2B}{2a_0}$.\\
Then zero is a global  attractor for all solutions of problem~(\ref{21}),(\ref{17}).
\end{Th}

\begin{Th}\label{th3.4}
 Consider the equation
\begin{equation}\label{22}
\ddot{x}(t)+f(t,x(t),\dot{x}(t))+\sum_{k=1}^m s_k(t,x(h_k(t)),\dot{x}(t))=0,
\end{equation}
where
$$
f(t,v,0)=0,  s_k(t,0,u)=0, 0<a_0\leq \frac{f(t,v,u)}{u}\leq A,
$$$$
 0<b_k\leq \frac{s_k(t,v,u)}{v}\leq B_k, u\neq 0,~ t-h_k(t)\leq \tau.
$$

Suppose at least one of the following conditions holds:\\
1) $\displaystyle \sum_{k=1}^m B_k\leq \frac{a_0^2}{4},~ \frac{a_0}{2}(A-a_0)<\sum_{k=1}^m b_k-a_0\sum_{k=1}^m B_k\tau_k$,\\
2) $\displaystyle \sum_{k=1}^m b_k\geq\frac{a}{2}\left(A-\frac{a_0}{2}\right),~ \sum_{k=1}^m
B_k(1+a_0\tau_k)<\frac{a_0^2}{2}$.\\
Then zero is a global  attractor for all solutions of problem~(\ref{22}),(\ref{17}).
\end{Th}

\begin{Th}\label{th3.5}
 Consider the equation
\begin{equation}\label{23}
\ddot{x}(t)+f(t,x(t),\dot{x}(t))+s(t,x(t))=\sum_{k=1}^m c_k(t)(x(t)-x(h_k(t))),
\end{equation}
where
$$
f(t,v,0)=0, s(t,0)=0, 0<a_0\leq \frac{f(t,v,u)}{u}\leq A,
$$$$
0<b_0\leq \frac{s(t,u)}{u}\leq B,~ |c_k(t)|\leq C_k,~ u\neq 0,~t-h_k(t)\leq \tau_k.
$$
Suppose at least one of the following conditions holds:\\
1) $\displaystyle B\leq \frac{a_0^2}{4},~ \sum_{k=1}^m C_k\tau_k<\frac{2b_0-a_0(A-a_0)}{2a_0}$,\\
2) $\displaystyle b_0 \geq \frac{a_0}{2}\left(A-\frac{a_0}{2}\right),~ \sum_{k=1}^m
C_k\tau_k<\frac{a_0^2-2B}{2a_0}$.\\
Then zero is a global  attractor for all solutions of problem~(\ref{23}),(\ref{17}).
\end{Th}

\begin{Ex}
To illustrate Part 2) of Theorem~\ref{th3.4}, consider the equation
\begin{equation}\label{ex1}
\ddot{x}(t)+(1.9+0.1\sin x(t))\dot{x}(t)+
(1.1+0.1\cos x(t))x(t-0.19 \sin^2 t)=0.
\end{equation}
We have
$m=1$, $a_0=1.8$, $A=2$, $b_0=1$, $B=1.2$, $\tau=0.19$;
therefore, all conditions of the theorem hold, hence  zero is a global  attractor for all solutions of equation~(\ref{ex1}).
\end{Ex}

Motivated by model \eqref{Busn}, consider a generalized  Kaldor-Kalecki model
\begin{equation}\label{ec1}
\ddot{x}(t)+ \left[ \alpha(t)-\beta(t)p'(x(t)) \right] \dot{x}(t)+s(t,x(t))=p(x(t))-p(x(h(t))),
\end{equation}
where $\alpha,\beta$ are locally essentially bounded functions, $s$ is a Caratheodory function,
$p$ is a locally absolutely continuous nondecreasing function,
$$
0<\alpha_0\leq \alpha(t)\leq \alpha_1,~ 0<\beta_0\leq \beta(t)\leq \beta_1,
$$$$
|p'(t)|\leq C,~ \alpha_0-\beta_1 C>0,~ 0<b_0\leq \frac{s(t,u)}{u}\leq B,~ t-h(t)\leq \tau.
$$
Denote $a_0=\alpha_0-\beta_1 C$.
\begin{Th}\label{th4.1}
Suppose at least one of the following conditions holds:\\
1) $\displaystyle B\leq \frac{a_0^2}{4},~ C\tau<\frac{2b_0-a_0(\alpha_1-a_0)}{2a_0}$,\\
2) $\displaystyle b \geq \frac{a_0}{2}\left(\alpha_1-\frac{a_0}{2}\right),~ C\tau<\frac{a_0^2-2B}{2a_0}$.\\
Then zero is a global  attractor for all solutions of problem~(\ref{ec1}),(\ref{17}).
\end{Th}
\begin{proof}
Suppose $x$ is a fixed solution of  problem~(\ref{ec1}),(\ref{17}).
There exists a function $\xi(t)$ such that $p(x(t))-p(h(x(t))=p'(\xi(t))(x(t)-x(h(t)))$.
Denote $\alpha(t)-\beta(t)p'(x(t))=a(t), p'(\xi(t))=c(t)$. Hence $x$ is a solution of the following equation
\begin{equation}\label{ec2}
\ddot{y}(t)+a(t)\dot{y}(t)+s(t,y(t))=c(t)(y(t)-y(h(t))).
\end{equation}
Since $p'(x)\geq 0$ then $0<\alpha_0-\beta_1 C\leq a(t)\leq \alpha_1$.
Equation~(\ref{ec2}) has a form (\ref{23}) with $f(t,x(t),\dot{x}(t))=a(t)\dot{x}(t), m=1$.
All conditions of Theorem~\ref{th3.5} hold, hence for any solution of (\ref{ec2}) we have $\lim_{t\rightarrow\infty}y(t)=0$. Then also
$\lim_{t\rightarrow\infty}x(t)=0$.
\end{proof}

\section{Sunflower model and its modifications}

The sunflower equation was introduced in 1967 by Israelson and Johnson in \cite{isa} as a model
for the geotropic circumnutations of {\em Helianthus annuus} and studied in \cite{alf,liza,somolinos}.
Historically, it was derived from the following first order delay equation
\begin{equation}
\label{Ori}
\dot{u}+\frac{b}{\tau} e^{a(1-t/\tau)}\int_{-\infty}^{t-\tau} e^{as/\tau}\sin u(s)ds=0.
\end{equation}
Taking the derivative of \eqref{Ori} we arrive at the sunflower equation
\begin{equation}
\label{Sun0}
\displaystyle  \ddot{x}+\frac{a}{ \tau} \dot{x}+\frac{b}{\tau}\, \sin x(t-\tau)=0,
\end{equation}
for which evidently the results of the previous section are not applicable.
\begin{Rem}
It is interesting to note that a non-delayed version of (\ref{Sun0})
\begin{equation}
\label{0_ode}
\displaystyle \ddot{x} +a \dot{x} +b \, \sin x(t)=0,
\end{equation}
has a long history, (see, for example, \cite{Mah}).
However, many important questions for delayed model (\ref{Sun0}) are still left unanswered.
\end{Rem}
Consider a generalization of model (\ref{Ori})
\begin{equation}
\label{add1}
~~~~~~\frac{du}{dt}+b \int_{-\infty}^{h(t)} K(t,s)\, \sin u(s)~ds =0,
\end{equation}
with the initial conditions
\begin{equation}
\label{add2}
u(t)=\varphi(t), ~ t \leq 0,
\end{equation}
under the following assumptions:\\
(b1) $h(t)\leq t-\tau$ for some $\tau > 0$;\\
(b2) $K(\cdot,\cdot)$ is  Lebesgue measurable,  $K(t,s) \geq 0$, there exists $a>0$ such that\\
  $\displaystyle K(t,s) \leq
\frac{1}{\tau} \exp\left\{ -\frac{a}{\tau} (t-s-\tau) \right\}$
and $\displaystyle \int_0^{\infty}  \int_{-\infty}^{h(t)} K(t,s)~ds~dt =\infty$;\\
(b3) $\varphi:[-\infty, 0]\to\R$ is a continuous bounded function.

\begin{Th}
\label{theorem_add1}
Suppose that (b1)-(b3) hold, $b>0$ and the characteristic equation
\begin{equation}
\label{add3}
\lambda^2 \tau - a \lambda +be^{\lambda \tau} =0
\end{equation}
has a positive root $\lambda_0>0$. Then any solution of (\ref{add1})-(\ref{add2}) with
the initial conditions satisfying either $\varphi(t) \in (2\pi k, 2\pi k +\pi)$, $k \in \N$,
or $\varphi(t) \in (2\pi k -\pi,2\pi k)$, $k \in \N$, together with
$|\varphi(t)-2\pi k| \leq \varphi(0) e^{-\lambda_0 t}$, $t<0$, tends to $2\pi k$ as $t \to \infty$.

Moreover, for $\varphi(t) \in (2\pi k, 2\pi k +\pi)$ the solution is monotone decreasing,
while for $\varphi(t) \in (2\pi k -\pi,2\pi k)$ it is monotone increasing.
\end{Th}
\begin{proof}
First assume that $\varphi(t) \in (0,\pi)$, $t \leq 0$, $u$ is a solution of \eqref{add1}.
{\bf
Let us prove that
\begin{itemize}
\item[(i)]
$u(t)$ is positive and non-increasing function;
\item[(ii)]
$u(t)$ satisfies the inequality
\begin{equation}
\label{add1_abc}
u(t) \geq u(0) e^{-\lambda_0 t}, \quad t \geq 0;
\end{equation}
\item[(iii)]
$u(t)$ tends to zero as $t \to \infty$.
\end{itemize}
Denote $u(t) = \varphi(t)$ for  $t \leq 0$ as well,
then by the assumptions of the theorem, $u(t)\leq u(0)e^{-\lambda_0 t}$, $t<0$.

We start verifying (ii) by induction. First, we prove that
$u(t) \geq u(0) e^{-\lambda_0 t}$ for $t \in [0,\tau]$, and then proceed to
any segment $[n\tau, (n+1)\tau]$.
In the inequalities below, we use the estimates of $K$ in (b2), the fact that $\sin u \leq u$ for $u>0$ and
$u(t)\leq u(0)e^{-\lambda_0 t}$ for $t<0$ to evaluate the derivative of $u$ on $[0,\tau]$:
\begin{align*}
\frac{du}{dt} = &-b \int_{-\infty}^{h(t)} K(t,s) \sin(u(s))~ds \geq -b \int_{-\infty}^{t-\tau} K(t,s) \varphi(s)~ds
\\
\geq & - \varphi(0) \frac{b}{\tau} \int_{-\infty}^{t-\tau} \exp\left\{ -\frac{a}{\tau} (t-s-\tau) \right\} e^{-\lambda_0 s} ds
\\
= & - \varphi(0) \frac{b}{\tau} \exp\left\{ -\frac{a}{\tau} (t-\tau)\right\}  \int_{-\infty}^{t-\tau}
\exp\left\{ \left( \frac{a}{\tau}-\lambda_0 \right)s
 \right\} ~ds
\\
= & - \varphi(0) \frac{b}{a-\lambda_0 \tau} e^{-\lambda_0 (t-\tau)} = - \varphi(0) \lambda_0 e^{-\lambda_0 t},
\end{align*}
since $a-\lambda_0 \tau = \frac{b}{\lambda_0} e^{-\lambda_0 \tau}$ by \eqref{add3}.

Since  $u'(t) \geq - u(0) \lambda_0  e^{-\lambda_0 t}$, for any $s$, $t$ satisfying $0\leq t \leq s \leq \tau$, the solution
on $[t,s]$ is not below the curve
$u(s)=u(t)  e^{-\lambda_0 (s-t)}$ on $[0,\tau]$, and $u(s) \geq u(t)  e^{-\lambda_0 (s-t)}$.
Taking $s=\tau$, we obtain
$$
u(\tau) \geq u(t) e^{-\lambda_0 (\tau-t)},\mbox{~~ or ~~~} u(t) \leq u(\tau) e^{-\lambda_0 (t-\tau)},
\quad t \in [0,\tau].
$$

Hence $u(\tau) e^{-\lambda_0 (t-\tau)}$ is an upper bound of $u(t)$  on $[0,\tau]$, it is also
a bound on $(-\infty,\tau]$ since $u(0) \leq u(\tau) e^{-\lambda_0 \tau}$ and for $t \leq 0$,
\begin{equation}
\label{tournesol1}
u(t) \leq  u(0) e^{-\lambda_0 t} \leq u(\tau) e^{-\lambda_0 (t-\tau)}.
\end{equation}
Thus $u(t) \leq u(\tau) e^{-\lambda_0 (t-\tau)}$ is valid on $(-\infty,\tau]$.

Consider further the initial problem with a shifted initial point $t_0=\tau$ instead of  $t_0=0$,
we get the same estimate as in (\ref{tournesol1}) for any $t \in (-\infty,n\tau]$ by induction.
Hence,
$$u(t) \geq u(n\tau)e^{-\lambda_0 (t-n\tau)}\geq u(0) e^{-\lambda_0 t}>0, ~~[n\tau, (n+1)\tau],$$
and the induction step proves (\ref{add1_abc}) and justifies (ii).

Thus the solution is positive for any $t$. From the form of the
equation and non-negativity of $K$ and $u \in (0,\pi)$ follow that the solution is also non-increasing,
which justifies (i).

Since $u$ is non-increasing for $t \geq 0$ and positive there is $\lim_{t \to \infty} u(t)=d$. Assuming $d>0$ we obtain
from $\displaystyle \int_0^{\infty}  \int_{-\infty}^{h(t)} K(t,s)~ds~dt =\infty$ in (b2) that $\lim_{t \to
\infty} u(t)=-\infty$, which is a contradiction, thus (iii) is also valid.
}

A similar process proves the case  $\varphi(t) \in (-\pi,0)$. If $\varphi(t) \in 2\pi k -\pi,2\pi k)$,
we apply the same argument to $u-2\pi k$.
\end{proof}

Note that sharp conditions when all solutions of characteristic equation (\ref{add3}) have positive real parts
can be found in \cite[Lemma 3.1, p. 470]{somolinos}.

\begin{Cor}
Let
\begin{equation}
\label{explicit}
\tau<\frac{a^2}{4b} e^{-a/2}
\end{equation}
and $|\varphi(t)-2\pi k| \leq \varphi(0) e^{-\lambda_0 t}$, $t<0$,
then any solution of (\ref{add1})-(\ref{add2}) with
the initial conditions satisfying  $\varphi(t) \in (2\pi k, 2\pi k +\pi)$, $k \in \N$, is monotone decreasing
and tends to $2\pi k$ as $t \to \infty$. Any solution with $\varphi(t) \in (2\pi k -\pi,2\pi k)$, $k \in \N$
tends to $2\pi k$ as $t \to \infty$.
\end{Cor}
\begin{proof}
Let $f(\lambda)=\tau\lambda^2 - a\lambda+b e^{\lambda\tau}$, then $f(0)=b>0$. Inequality
(\ref{explicit}) implies $f(a/(2\tau))=-a^2/(4 \tau) +b e^{a/2}<0$, so equation (\ref{add3}) has a positive solution.
We invokef Theorem~\ref{theorem_add1} to conclude the proof.
\end{proof}

The following example illustrates that conditions (b1)-(b3) do not guarantee boundedness of the solutions of equation (\ref{add1}) with the
generalized kernel.

\begin{Ex}
\label{tournesol_ex}
Let $\displaystyle a=\frac{1}{3} \ln \left( \frac{4}{\pi} \right)$, $b=2$, $\tau=\pi$,
$$ K(t,s) = \left\{  \begin{array}{lll} {\displaystyle \frac{1}{4}}, & t \in [(2k-1)\pi, (2k+1)\pi],
& s \in [(2k-3)\pi,(2k-2)\pi], \\
0, & t \in [(2k-1)\pi, (2k+1)\pi],
& s \not\in [(2k-3)\pi,(2k-2)\pi].  \end{array}  \right.
$$
Then obviously $K(t,s)=0$ for $s>t-\pi=t-\tau$, and also for $t-s>4\pi$. The exponential estimate has the form
$$0 \leq K(t,s) \leq \frac{1}{\pi} e^{-\frac{1}{3\pi}\ln(4/\pi)(t-s-\pi)} =\frac{1}{\pi} \left( \frac{4}{\pi}
\right)^{-(t-s-\pi)/(3\pi)},
$$
but as $t-s-\pi \leq 3\pi$ whenever $K(t,s) \neq 0$, the right-hand side is not less than $\displaystyle
\frac{1}{\pi} \left( \frac{4}{\pi}\right)^{-1}=\frac{1}{4}$, thus $K(t,s)$ has an exponential estimate as in
(b2). Further, $u(t)=t$ is an unbounded solution of (\ref{add1}). In fact, let $u(t)=t$, $t \in [-\pi,\pi]$.
Then for $t\in [\pi, 3\pi]$ we have
$\displaystyle \frac{du}{dt} = -2 \int_{-\pi}^0 \frac{1}{4} \sin(t) ~dt = 1,$
so $u(t)=t$ on $[-\pi,3\pi]$. Due to the periodicity of the sine function and $K$, we have $\displaystyle \frac{du}{dt} \equiv
1$.
Thus the solution is a linear function  $u(t)=t$ and it is unbounded.
\end{Ex}

In the following theorem we will prove that for a non-autonomous case the solution
of the sunflower equation is bounded by a linear function.

Consider the non-autonomous sunflower equation
\begin{equation}\label{gsun}
\ddot{x}(t)+a(t)\dot{x}(t)+b(t)\sin x(h(t))=0.
\end{equation}
\begin{Th}
\label{tt} Suppose $a(t)\geq a_0>0, |b(t)|\leq b_0$.
For any solution $x(t)$ of equation~(\ref{gsun}) we have the estimates
\begin{equation}\label{gsun1}
|x(t)|\leq |x(t_0)|+\left(|\dot{x}(0)|+\frac{b_0}{a_0}\right)t, ~~
|\dot{x}(t)|\leq |\dot{x}(0)|+\frac{b_0}{a_0}.\nonumber
\end{equation}
\end{Th}
\begin{proof}
Denote $\dot{x}=y, f(t)=b(t)\sin x(h(t))$, where $|f(t)|\leq b_0.$
Then $\dot{y}(t)+a(t)y(t)+f(t)=0$, hence $y(t)=y(0)+\int_0^t e^{-\int_s^t a(\tau)d\tau}f(s)ds$.
Then
 $$
|\dot{x}(t)|\leq |\dot{x}(0)|+\int_0^t e^{-a_0(t-s)}|f(s)ds|\leq |\dot{x}(0)|+\frac{b_0}{a_0},
$$$$
x(t)=x(0)+\int_0^t \dot{x}(s)ds, ~|x(t)|\leq |x(t_0)|+\left(|\dot{x}(0)|+\frac{b_0}{a_0}\right)t.
$$
Local stability conditions for equation~(\ref{gsun}) one can find in the following theorem.
\end{proof}

{\bf
The following lemma is a corollary of \cite[Theorem 8.3]{ABBD}
\begin{Lem}
\label{theorem8.3}
Suppose that ordinary differential equation
$$
\ddot{x}(t)+a(t)\dot{x}(t)+b(t)x(t), ~~a(t)\geq 0, b(t)\geq 0,
$$
has a positive fundamental function, then the equation
 $$
\ddot{x}(t)+a(t)\dot{x}(t)+b_1(t)x(t),
$$
where $b_1(t)\leq b(t)$ also has a positive fundamental function.
\end{Lem}
}

\begin{Th}
Suppose $0<a \leq a(t)\leq A$, $0<b \leq b(t)\leq B,~ t-h(t)\leq \tau$
and at least one of the following conditions hold:\\
1) $\displaystyle  B\leq \frac{a^2}{4},~ \frac{a}{2}(A-a)< b-a B\tau$,\\
2) $\displaystyle  b\geq \frac{a}{2}\left(A-\frac{a}{2}\right),~
B\left(1+a\tau\right)<\frac{a^2}{2}$.\\ \\
Then any equilibrium $x(t)=2k\pi, k=0,\dots$
of equation~(\ref{gsun}) is locally asymptotically stable. Any equilibrium $x(t)=(2k+1)\pi, k=0,\dots$ is not asymptotically stable.
\end{Th}
\begin{proof}
For the equilibrium $x(t)=2k\pi$, the linearization of equation~(\ref{gsun}) has the form
$$
\ddot{y}(t)+a(t)\dot{y}(t)+b(t)y(h(t))=0,
$$
which is exponentially stable by Theorem~\ref{th3}.

{\bf 
It is well known (see, for example, \cite{B3})
that exponential stability of a linearized equation implies asymptotic stability of the nonlinear equation,
in our case equation~(\ref{gsun}).
}

For the equilibrium $x(t)=(2k+1)\pi$, the linearized equation for (\ref{gsun})
has the form
\begin{equation}\label{s1}
\ddot{y}(t)+a(t)\dot{y}(t)-b(t)y(h(t))=0.
\end{equation}
Consider now the ordinary differential equation
\begin{equation}\label{s2}
\ddot{z}(t)+a(t)\dot{z}(t)=0.
\end{equation}
The fundamental function of equation~(\ref{s2}) (the solution of initial value problem with $z(0)=0, z'(0)=1$)
has the form
$$
z(t)=\int_{0}^t e^{-\int_0^s a(\tau)d\tau}ds,
$$
which is a positive function for $t>0$ with a nonnegative derivative.

By Lemma~\ref{theorem8.3},
for the  fundamental function $y(t)$ of equation~(\ref{s1}) we have $y(t)>0$,
$y'(t)\geq 0$ for $t>0$.
Hence $y(t)$ does not tend to zero, and thus equation~(\ref{s1}) is not asymptotically stable.
\end{proof}

\section{Concluding Remarks}

The technique of reduction of a high-order linear differential equation to a system by the substitution $x^{(k)}=y_{k+1}$ is quite common.
However, this substitution does not depend on the parameters of the original equation, and therefore does not
offer new insight from a qualitative analysis point of view.
Instead, we proposed  a substitution which exploits the parameters of the original model.
By using that approach, a broad class of the second order non-autonomous linear equations with delays was
examined and explicit easily-verifiable sufficient stability conditions were obtained. There is a natural
extension of this approach to stability analysis of high-order models. For the nonlinear second order
non-autonomous equations with delays we applied the linearization technique and the results obtained for
linear models. Our stability tests  are applicable to some milling models, e.g.  models (\ref{stepa}) and
(\ref{L}), and to a non-autonomous Kaldor--Kalecki business cycle model.
Several numerical examples illustrate the application of the stability tests.
We suggest that a similar technique can be developed for higher order linear delay equations,
with or without non-delay terms. For a non-autonomous version of a classical sunflower model, we verified that the derivative is bounded and
thus the solution has a linear bound.
Example~\ref{tournesol_ex} illustrates the existence of an unbounded linearly growing solution for
the generalized sunflower equation. We also obtained sufficient conditions under which a solution tends to one of
the infinite number of the equilibrium points.

Solution of the following problems will complement the results of the present paper:
\begin{enumerate}
\item
In all stability conditions obtained, we used lower and upper bounds of the coefficients and the delays.
It is interesting to obtain stability conditions in an integral form, for instance, in the assumptions of
Theorem~\ref{th3} replace the term $a \tau_k$  by, generally, a smaller term $\int_{h_k(t)}^t a(s)~ds$.
\item
Apply the technique used in the paper to examine  delay  differential  equations of higher order.
\item
Is it possible to generalize Theorem~\ref{theorem_add1} to the case when the initial function $\varphi(t) \in (2\pi k-\pi, 2\pi
k+\pi)$ and characteristic equation (\ref{add3})
has a solution with a  positive real part?
\item
Establish necessary stability conditions for the equations considered in this paper.
\item
For the sunflower equation and its modifications establish set of conditions to guarantee boundedness of
all solutions.
\end{enumerate}

\section{Acknowledgments}
The authors also would like to extend their appreciation to the anonymous referees for helpful suggestions which have greatly improved this paper.\\

L. Berezansky was partially supported by Israely Ministry of Absorption,
E. Braverman was partially supported by the NSERC grant RGPIN/261351-2010,
L. Idels was partially supported by a grant from VIU.
{\bf   The authors are grateful to the anonymous reviewer whose comments significantly 
contributed to the presentation of the paper.}

\end{document}